\documentclass[12pt]{amsart}
\usepackage{epsfig,amsmath,amsfonts,amssymb,amsthm}
\usepackage[all]{xy}
\usepackage{amscd}
\usepackage{amsmath, pb-diagram}
\usepackage{enumerate}
\newtheorem{thm}{Theorem}[section]
\newtheorem{lem}[thm]{Lemma}
\newtheorem{prop}[thm]{Proposition}
\newtheorem{cor}[thm]{Corollary}
\newtheorem{defn}[thm]{Definition}

\begin{document}

\title[$C(\mathcal{X^{*}})$-Cover and $C(\mathcal{X^{*}})$-Envelope ]
{$C(\mathcal{X^{*}})$-Cover and $C(\mathcal{X^{*}})$-Envelope }

\author{Tah\.{i}re \" Ozen}
\author{Em\.{i}ne Y{\i}ld{\i}r{\i}m}
\address{Department of Mathematics, Abant \.{I}zzet Baysal University
\newline
G\"olk\"oy Kamp\"us\"u Bolu, Turkey} \email{ozen\_t@ibu.edu.tr}
\email{emineyyildirim@gmail.com}

\date{\today}
\subjclass{18G35} \keywords{Complex, Cover, Envelope, Inverse and
Direct limit.} \pagenumbering{arabic}

\begin{abstract} Let $R$ be any associative ring with unity and
$\mathcal{X}$ be a class of $R$-modules of closed under direct sum (and summands) and with extension closed. We prove that every complex has an $C(\mathcal{X^{*}})$-cover
($C(\mathcal{X^{*}})$-envelope) if every module has an
$\mathcal{X}$-cover ($\mathcal{X}$-envelope) where
$C(\mathcal{X^{*}})$ is the class of complexes of modules in
$\mathcal{X}$ such that it is closed under direct and inverse limit. 
\end{abstract}

\maketitle

\section{Introduction}
\begin{defn}($\mathcal{X^{*}}-complex$ ) Let $\mathcal{X}$ be a class of R-modules. A
complex $\mathcal{C}: \ldots \longrightarrow C^{n-1} \longrightarrow
C^{n} \longrightarrow C^{n+1} \longrightarrow \ldots$ is called an
$\mathcal{X^{*}}-(cochain)$ complex if  $C^i \in \mathcal{X}$ for
all $i \in \mathbb{Z}$. A complex $\mathcal{C}: \ldots
\longrightarrow C_{n+1} \longrightarrow C_{n} \longrightarrow
C_{n-1} \longrightarrow \ldots$ is called an
$\mathcal{X^{*}}-(chain)$ complex if $C_i \in \mathcal{X}$ for all
$i\in \mathbb{Z}$. The class of all $\mathcal{X^{*}}-complexes$ is
denoted by $C(\mathcal{X^{*}})$.
\end{defn}

\begin{defn}Let $\mathcal{A}$ be an abelian category and let $\mathcal{Y}$ be
a class of objects of $\mathcal{A}$. Then for an object \;$M \in
\mathcal{A}$\;a morphism \;$\phi: M \longrightarrow X$\; where $X
\in \mathcal{Y}$ is called an \;$\mathcal{Y}$-preenvelope of $M$, if
\;$f:M \longrightarrow  X'$\; where \;$X' \in \mathcal{Y}$\; and the
following diagram

\[\begin{diagram}
\node{M} \arrow{e,t}{f} \arrow{s,t}{\phi} \node{X'}\\
\node{X} \arrow{ne,r,..}{g}
\end{diagram}\] \\
\noindent can be completed to a commutative diagram.

\noindent If when $X=X'$ and the only such $g$ is automorphism of
$X$, then \;$\phi: M \longrightarrow X$\; is called an
\;$\mathcal{Y}$-envelope of $M$. Dually it can be given an
\;$\mathcal{Y}$-precover (cover)of $M$.
\end{defn}

\begin{defn}($\#$-injective complex ) (\cite{AF})A complex of
injective modules is called $\#$-injective complex.
\end{defn}

\noindent In \cite{AA} and \cite{A},respectively, it is proved that
if $R$ is a noetherian ring, then every complex of  $R$-modules has
a $\#$-injective cover and if $R$ is a Gorenstein ring, then every
bounded complex has a Gorenstein projective precover. Using the
proof of Lemma 7 in \cite{A}, we give a generalization of \cite{AA}
and \cite{A}. (See \cite{R} and \cite{EJ} for the other definitions
and concepts.) Throughout the paper $\mathcal{X}$ denotes a class of
$R$-modules.

\section{General Results}

\begin{lem} Let for all $n \in \mathbb{Z}$, $0 \longrightarrow P_{n} \longrightarrow
X_{n}$ is an $\mathcal{X}-precover$ of $X_{n}$. Then we have a monic
chain map \;$0 \longrightarrow P \longrightarrow X$\; where $P:...\rightarrow
P_{2}\rightarrow P_{1}\rightarrow P_{0}\rightarrow ...$ and $X:...\rightarrow
X_{2}\rightarrow X_{1}\rightarrow X_{0}\rightarrow ...$ .
\end{lem}

\begin{proof} It is trivial.
\end{proof}

\begin{thm} Let every module have a monic $\mathcal{X}-precover$(or an epic $\mathcal{X}-preenvelope$)
then every complex $X$ has a monic $C(\mathcal{X^{*}})-precover$(or
epic $C(\mathcal{X^{*}})-preenvelope$). Moreover if every module has
a monic $\mathcal{X}-cover$(or epic $\mathcal{X}-envelope$), then
every complex $X$ has a monic $C(\mathcal{X^{*}})-cover$ (or epic
$C(\mathcal{X^{*}})-envelope$).
\end{thm}

\begin{proof} Let X be any complex and $0\rightarrow P_{n}\rightarrow X_{n}$ be an $\mathcal{X}$-precover of $X_{n}$.
By Lemma 2.1 there exists a complex $P:...\rightarrow
P_{2}\rightarrow P_{1}\rightarrow P_{0}\rightarrow ...$ such that
$0\rightarrow P \rightarrow X$ is a chain map. Let
$P':...\rightarrow P'_{2}\rightarrow P'_{1}\rightarrow
P'_{0}\rightarrow ...$ be any $\mathcal{X^{*}}$-complex and $f':P'
\longrightarrow X$ be a chain map. Then we have the following
diagram as follow:
\[\begin{diagram}
\node[2]{P_{n+1}} \arrow{e,t} {\lambda_{n+1}} \arrow{sw,t} {f_{n+1}}
\node{P_{n}} \arrow{e,t} {\lambda_{n}} \arrow{sw,t} {f_{n}}
\node{P_{n-1}} \arrow{sw,t}{f_{n-1}} \\
\node{X_{n+1}} \arrow{e,t}{\beta_{n+1}} \node{X_{n}}
\arrow{e,t}{\beta_{n}}
\node{X_{n-1}} \\
\node[2]{P'_{n+1}} \arrow{e,r} {\lambda'_{n+1}} \arrow{nw,r}
{f'_{n+1}} \node{P'_{n}} \arrow{e,r} {\lambda'_{n}} \arrow{nw,r}
{f'_{n}} \node{P'_{n-1}} \arrow{nw,r} {f'_{n-1}} \\
\end{diagram}\]
\noindent where since $P'$ is an $\mathcal{X^{*}}-complex$ and $f':P'
\longrightarrow X$ be a chain map, we can find $\theta_{n}:P'_{n}
\longrightarrow P_{n}$ for all $n \in \mathbb{Z}$ such that
${f_{n}}{\theta_{n}}={f'_{n}}$.

\noindent Since $f'$ is a chain map,
${f'_{n}}{\lambda'_{n+1}}={\beta_{n+1}}{f'_{n+1}}$, for all $n \in
\mathbb{Z}$. Now we will show that $\theta$ is a chain map. Since
$f_{n}$ is a chain map,

${f_{n}}{\lambda_{n+1}}={\beta_{n+1}}{f_{n+1}}$,
${f_{n}}{\lambda_{n+1}}{\theta_{n+1}}={\beta_{n+1}}{f_{n+1}}{\theta_{n+1}}={\beta_{n+1}}{f'_{n+1}}
={f'_{n+1}}{\lambda'_{n+1}}={f_{n}}{\theta_{n}}{\lambda'_{n+1}}$

\noindent Since $f_{n}$ is $1-1$,
${\lambda_{n+1}}{\theta_{n+1}}={\theta_{n}}{\lambda'_{n+1}}$. So
$\theta$ is a chain map. Let $0 \longrightarrow P \longrightarrow X$ be a monic chain map
such that $0 \longrightarrow P_{n}\longrightarrow X_{n}$ is an
$\mathcal{X}-cover$ of $X_{n}$. Then $0 \longrightarrow P
\longrightarrow X$ is an $C(\mathcal{X^{*}})-cover$ of $X$. Because
we have the following diagram:
\[\begin{diagram}
\node[2]{P_{n+1}} \arrow{e,t} {\lambda_{n+1}} \arrow{sw,t} {f_{n+1}}
\node{P_{n}} \arrow{e,t} {\lambda_{n}} \arrow{sw,t} {f_{n}}
\node{P_{n-1}} \arrow{sw,t}{f_{n-1}} \\
\node{X_{n+1}} \arrow{e,t}{\beta_{n+1}} \node{X_{n}}
\arrow{e,t}{\beta_{n}}
\node{X_{n-1}} \\
\node[2]{P_{n+1}} \arrow{e,r} {\lambda_{n+1}} \arrow{nw,r}
{f_{n+1}} \node{P_{n}} \arrow{e,r} {\lambda_{n}} \arrow{nw,r}
{f_{n}} \node{P_{n-1}} \arrow{nw,r} {f_{n-1}} \\
\end{diagram}\]
\noindent We can find that for all $n\in \mathbb{Z}$, $\theta_{n}$ is an
automorphism. By the proof above $\theta:P \rightarrow P$ is a chain
map such that ${f}{\theta}={f}$. Since for all $n\in \mathbb{Z}$,
$\theta_{n}$ is a automorphism, $\theta$ is an automorphism with
${f}{\theta}={f}$. So $0 \longrightarrow P \longrightarrow X$ is a
cover of $X$. Dually we can find a
$C(\mathcal{X^{*}})-preenvelope(envelope)$; $X \longrightarrow E
\longrightarrow 0$.

\end{proof}

\begin{lem} \label{1.28} Let $\mathcal{X}$ be closed under finite sum. If every module has an $\mathcal{X}$-precover, then
every bounded complex has an $C(\mathcal{X^{*}})$-precover.
\end{lem}

\begin{proof} Let $... \longrightarrow E_{1}' \overset{f_{1}} \longrightarrow E_{0}' \overset{\phi_{0}} \longrightarrow
X_{0} \longrightarrow 0$ be a minimal $\mathcal{X}$-precover
resolvent of $X_{0}$. Let A be an $\mathcal{X^{*}}$-complex and
$\beta \in Hom(A,X(0))$ where \;$X(0):... \longrightarrow 0
\longrightarrow X_{0} \longrightarrow 0 \longrightarrow 0
\longrightarrow ...$.
\noindent We have the diagram as follow:
\[\begin{diagram}
\node{...} \arrow{e} \node{A_{1}} \arrow{e,t} {a_{1}} \arrow{s,r}
{\lambda_{1}} \node{A_{0}} \arrow{e} \arrow{s,r} {\lambda_{0}}
\node{A_{-1}} \arrow{e}
\arrow{s} \node{...} \\
\node{...} \arrow{e} \node{E_{1}'} \arrow{e,t} {f_{1}} \node{E_{0}'}
\arrow{e} \arrow{s,r} {\phi_{0}} \node{0} \\
\node[2] {0} \arrow{e} \node{X_{0}}
\arrow{e} \node{0} \\
\end{diagram}\]
\noindent where ${\phi_{0}}{\lambda_{0}}={\beta_{0}}$ and
${\lambda_{0}}{a_{1}} \in Hom(A_{1},Ker\phi_{0})$. Then there exists
$\lambda_{1} \in Hom(A_{1},E_{1}')$, since $E_{1}' \longrightarrow
Ker\phi_{0}$ is a precover and similarly we can define $\lambda_{n}
\in Hom(A_{n},E_{n}')$ such that
$f_{n}\lambda_{n}=\lambda_{n-1}a_{n}$. Thus, $E':... \longrightarrow
E_{2}' \longrightarrow E_{1}' \longrightarrow E_{0}' \longrightarrow
0 \longrightarrow 0 \longrightarrow...$ is a
$C(\mathcal{X^{*}})$-precover of $X(0)$.
\noindent Let $X(n):... 0 \longrightarrow X_{n} \overset{\ell_{n}}
\longrightarrow ... \longrightarrow X_{1} \overset{\ell_{1}}
\longrightarrow X_{0} \longrightarrow 0 \longrightarrow 0
\longrightarrow ...$. We will use an induction on n to prove that
$X(n)$ has a $C(\mathcal{X^{*}})$-precover. We proved that $X(0)$
has a $C(\mathcal{X^{*}})$-precover above. Suppose that $D(n):...
\longrightarrow D_{n+1} \overset{d_{n+1}} \longrightarrow D_{n}
\longrightarrow .... \longrightarrow D_{1} \overset{d_{1}}
\longrightarrow D_{0} \longrightarrow 0$ be a
$C(\mathcal{X^{*}})$-precover of $X(n)$ such that
\[\begin{diagram}
\node{D_{n+1}} \arrow{e,t} {d_{n+1}}
\node{D_{n}} \arrow{e,t} {d_{n}} \arrow{s,r}
{\phi_{n}} \node{...} \arrow{e} \node{D_{1}} \arrow{e} \arrow{s,r} {\phi_{1}} \node{D_{0}} \arrow{e} \arrow{s,r} {\phi_{0}} \node{0}\\
\node{0} \arrow{e} \node{X_{n}}
\arrow{e,t} {\ell_{n}}
\node{...} \arrow{e} \node{X_{1}} \arrow{e,t} {\ell_{n}} \node{X_{0}} \arrow{e} \node{0}\\
\end{diagram}\]
\noindent Let $E_{\bullet}=... \rightarrow E_{2} \overset{t_{2}} \rightarrow
E_{1} \overset{t_{1}} \rightarrow E_{0} \rightarrow 0 \rightarrow 0
\rightarrow ...$ is a precover of $\underline{X_{n+1}}$ where
$t_{0}:E_{0} \rightarrow X_{n+1}$ is a precover of $X_{n+1}$. Then
we have the following commutative diagram:
\[\begin{diagram}
\node{E_{\bullet}} \arrow{e,t} {\nu} \arrow{s} \node{D(n)} \arrow{s,r} {\phi} \\
\node{\underline{X_{n+1}}} \arrow{e} \node{X(n)}\\
\end{diagram}\]
\noindent where $\ell_{n+1}: X_{n+1} \longrightarrow X_{n}$. Since
$\phi:D(n) \longrightarrow X(n)$ is a $C(\mathcal{X^{*}})$-precover
of X(n) and $E_{\bullet} \in C(\mathcal{X^{*}})$, we have a morphism
$\nu: E_{\bullet} \longrightarrow D(n)$ where
\[\begin{diagram}
\node{E_{\bullet}=...} \arrow{e} \node{E_{2}} \arrow{e,t} {t_{2}}
\arrow{s,r} {\nu_{2}} \node{E_{1}} \arrow{e,t} {t_{1}} \arrow{s,r}
{\nu_{1}} \node{E_{0}} \arrow{e} \arrow{s,r} {\nu_{0}}  \node{...} \\
\node{D(n)=...} \arrow{e} \node{D_{n+2}} \arrow{e,t} {d_{n+2}}
\node{D_{n+1}} \arrow{e,t}
{d_{n+1}} \node{D_{n}} \arrow{e,t} {d_{n}} \node{...} \\
\end{diagram}\]...(1)

\noindent where ${\phi_{n}}{\nu_{0}}={\ell_{n+1}}{t_{0}}$...(*).

\noindent Let C be the mapping cone of $\nu:E_{\bullet} \longrightarrow D(n)$.
So we have that \;$C=... \longrightarrow {D_{n+3} \oplus E_{2}}
\overset{\lambda_{n+3}} \longrightarrow {D_{n+2} \oplus E_{1}}
\overset{\lambda_{n+2}} \longrightarrow {D_{n+1} \oplus E_{0}}
\overset{\lambda_{n+1}} \longrightarrow D(n) \overset{d_{n}}
\longrightarrow D_{n-1} \longrightarrow ... \longrightarrow D_{1}
\longrightarrow 0$ where
$$\lambda_{n+1}(x,y)=d_{n+1}(x)+\nu_{0}(y)$$ and for $k \geq 2$
$$\lambda_{n+k}(x,y)=(d_{n+k}(x)+\nu_{k-1}(y),-t_{k-1}(y)).$$
\noindent We show that $C \longrightarrow X(n+1)$ is an
$C(\mathcal{X^{*}})$-precover of $X(n+1)$. Let $A \in
C(\mathcal{X^{*}})$ where
\[\begin{diagram}
\node{A=...} \arrow{e} \node{A_{n+2}} \arrow{e,t} {a_{n+2}}
\arrow{s} \node{A_{n+1}} \arrow{e,t} {a_{n+1}} \arrow{s,r}
{\beta_{n+1}} \node{A_{n}} \arrow{e,t} {a_{n}} \arrow{s,r}
{\beta_{n}}
\node{...} \\
\node{X(n+1)=...} \arrow{e} \node{0} \arrow{e} \node{X_{n+1}}
\arrow{e,t}
{\ell_{n+1}} \node{X_{n}} \arrow{e,t} {\ell_{n}} \node{...} \\
\end{diagram}\]
Then we have the following diagram:
\[\begin{diagram}
\node{...A_{n+2}} \arrow{e,t} {a_{n+2}}
\arrow{s,r} {\theta_{n+2}} \node{A_{n+1}} \arrow{e,t} {a_{n+1}}
\arrow{s,r} {\theta_{n+1}} \node{A_{n}} \arrow{e,t} {a_{n}}
\arrow{s,r} {\theta_{n}}
\node{...} \\
\node{...D_{n+2} \oplus E_{1}} \arrow{e,t}
{\lambda_{n+2}} \node{D_{n+1} \oplus E_{0}} \arrow{e,t}
{\lambda_{n+1}} \arrow{s,r} {(0,t_{0})} \node{D_{n}} \arrow{e,t}
{d_{n}} \arrow{s,r} {\phi_{n}} \node{...} \\
\node{...0} \arrow{e} \node{X_{n+1}} \arrow{e,t}{\ell_{n+1}} \node{X_{n}} \arrow{e,t}
{\ell_{n}} \node{...} \\
\end{diagram}\]
\noindent Firstly, we will find $\theta_{n+1}$ such that
${\lambda_{n+1}}{\theta_{n+1}}={\theta_{n}}{a_{n+1}}$ and
${(0,t_{0})}{\theta_{n+1}}={\beta_{n+1}}$. We know that
${\phi_{n}}{\lambda_{n+1}}={\ell_{n+1}}(0,t_{0})$ by (*). Since
$\beta_{n+1}:A_{n+1} \longrightarrow X_{n+1}$ and $A_{n+1} \in
\mathcal{X}$ and $t_{0}:E_{0} \longrightarrow X_{n+1}$ is an
$\mathcal{X}$-precover of $X_{n+1}$, there exists $r: A_{n+1}
\longrightarrow E_{0}$ such that $\beta_{n+1}={t_{0}}{r}$.

\noindent Claim: Let \;$L^{k}=Ker(D(k)\rightarrow X(k))$\; where $k \in
\{0,1,...,n\}$ and then $$L^{0}:... \longrightarrow E_{2}'
\overset{f_{2}} \longrightarrow E_{1}' \overset{f_{1}}
\longrightarrow Ker{\phi_{0}} \longrightarrow 0\longrightarrow...,$$
$L^{n}:... \longrightarrow D_{n+2} \overset{d_{n+2}}
\longrightarrow D_{n+1} \overset{d_{n+1}} \longrightarrow
Ker\phi_{n} \overset{d_{n}} \longrightarrow Ker\phi_{n-1}
\longrightarrow ... \longrightarrow Ker\phi_{1} \longrightarrow
Ker\phi_{0} \longrightarrow 0\longrightarrow...$ and also
$L^{n+1}:... \longrightarrow {D_{n+2} \oplus E_{1}}
\overset{\lambda_{n+2}} \longrightarrow {D_{n+1} \oplus Kert_{0}}
\overset{\lambda_{n+1}} \longrightarrow Ker\phi_{n} \longrightarrow
Ker\phi_{n-1} \longrightarrow ... \longrightarrow Ker\phi_{0}
\longrightarrow 0\longrightarrow...$ where $E_{\bullet}=...
\longrightarrow E_{2} \overset{t_{2}} \longrightarrow E_{1}
\overset{t_{1}} \longrightarrow E_{0} \longrightarrow 0$ (where
$E_{0}$ is in the nth place) is an $C(\mathcal{X^{*}})$-precover of
$\underline{X_{n+1}}$. We claim that for every $S \in
C(\mathcal{X^{*}})$, $Hom(S,L^{k})$ is exact where $k \in
\{0,1,...,n,n+1\}$.

\noindent Since $L^{0}:... \longrightarrow E_{2}' \overset{f_{2}}
\longrightarrow E_{1}' \overset{f_{1}} \longrightarrow Ker{\phi_{0}}
\longrightarrow 0$, for every $\mathcal{X^{*}}$-complex S,
$... \longrightarrow Hom(S,E_{2}') \longrightarrow Hom(S,E_{1}')
\longrightarrow Hom(S,Ker\phi_{0}) \longrightarrow 0$ is exact. So,
$Hom(S,L^{0})$ is exact. It is enough to prove that if
$Hom(S,L^{n})$ is exact, then \;$Hom(S,L^{n+1})$ is exact.
\noindent We have the following diagram between \;$L^{n}$ and $L^{n+1}$:
\[\begin{diagram}
\node{...D_{n+2}} \arrow{e,t} {d_{n+2}}
\arrow{s,r} {(1,0)} \node{D_{n+1}} \arrow{e,t} {d_{n+1}} \arrow{s,r}
{(1,0)} \node{Ker\phi_{n}} \arrow{e,t} {d_{n}} \arrow{s} \node{...}\\
\node{...D_{n+2} \oplus E_{1}} \arrow{e,t}
{\lambda_{n+2}} \node{D_{n+1} \oplus Kert_{0}} \arrow{e,t}
{\lambda_{n+1}} \node{Ker\phi_{n}} \arrow{e,t} {d_{n}} \node{...}\\
\end{diagram}\]
\noindent Therefore we have a split exact sequence $0 \longrightarrow L^{n}
\longrightarrow L^{n+1} \longrightarrow \frac{L^{n+1}}{L^{n}}
\longrightarrow 0$ where $\frac{L^{n+1}}{L^{n}} \cong ....
\longrightarrow E_{2} \longrightarrow E_{1} \longrightarrow Kert_{0}
\longrightarrow 0$ which is $Hom(\mathcal{X^{*}}-complex,-)$ exact
since $... \longrightarrow E_{2} \longrightarrow E_{1}
\longrightarrow E_{0} \longrightarrow X_{n+1} \longrightarrow 0$ is
a minimal $\mathcal{X}$-resolvent. Since $Hom(S,L^{n})$ and $Hom(S,\frac{L^{n+1}}{L^{n}})$ are exact,
so $Hom(S,L^{n+1})$ is exact.\\
\noindent Since $Hom(A_{n+1},L^{n})$ is exact where $A_{n+1} \in \mathcal{X}$,
$Hom(A_{n+1},D_{n+1}) \overset{d_{n+1}^{*}}
\longrightarrow Hom(A_{n+1},Ker\phi_{n}) \overset{d_{n}^{*}}
\longrightarrow Hom(A_{n+1},Ker\phi_{n-1})$ is
exact, that is $Kerd_{n}^{*}=Imd_{n+1}^{*}$. Thus \;$\phi_{n}({\theta_{n}}{a_{n+1}}-{\nu_{0}}{r})={\phi_{n}}{\theta_{n}}{a_{n+1}}-{\phi_{n}}{\nu_{0}}{r}={\beta_{n}}{a_{n+1}}-{\phi_{n}}{\lambda_{n+1}(0,r)}={\ell_{n+1}}{\beta_{n+1}}-{\ell_{n+1}(0,t_{0})(0,r)}={\ell_{n+1}}{\beta_{n+1}}-{\ell_{n+1}}{\beta_{n+1}}=0$. So, ${\theta_{n}}{a_{n+1}}-{\nu_{0}}{r} \in Ker\phi_{n}$ and also \;$d_{n}^{*}({\theta_{n}}{a_{n+1}}-{\nu_{0}}{r})=d_{n}({\theta_{n}}{a_{n+1}}-{\nu_{0}}{r})={d_{n}}{\theta_{n}}{a_{n+1}}-{d_{n}}{\nu_{0}}{r}={\theta_{n-1}}{a_{n}}{a_{n+1}}-0=0$.\\
\noindent Therefore \; ${\theta_{n}}{a_{n+1}}-{\nu_{0}{r}} \in Imd_{n+1}^{*}$; that is; there exists $s \in$\\ $Hom(A_{n+1},D_{n+1})$ such that
$${\theta_{n}}{a_{n+1}}-{\nu_{0}{r}}={d_{n+1}}{s}...(**)$$
\noindent Let $$\theta_{n+1}: A_{n+1} \longrightarrow {D_{n+1} \oplus E_{0}}$$
$$a \longrightarrow (s(a),r(a))$$
\noindent Then ${(0,t_{0})}{\theta_{n+1}}(a)=(0,t_{0})(s(a),r(a))$ where $a
\in A_{n+1}$ $$={t_{0}}{r(a)}$$ $$=\beta_{n+1}(a)$$
\noindent Therefore, ${(0,t_{0})}{\theta_{n+1}}=\beta_{n+1}$. Moreover, ${\lambda_{n+1}}{\theta_{n+1}}={\lambda_{n+1}}(s,r)={d_{n+1}}{s}+{\nu_{0}}{r}={\theta_{n}}{a_{n+1}}$(by (**) and diagram (1)).\\
\noindent Secondly, let's find $\theta_{n+2}$. We know that
${(0,t_{0})}{\lambda_{n+2}}=0$. Since
$(0,t_{0})({\theta_{n+1}}{a_{n+2}})=0$ and $t_{1}: E_{1}
\longrightarrow Kert_{0}$ is an $\mathcal{X}$-precover, there exists
$r_{1}:A_{n+2} \longrightarrow E_{1}$ such that
${t_{1}}{r_{1}}={r}{a_{n+2}}$. Since $Hom(A_{n+2},L^{n})$ is exact,
$$Hom(A_{n+2},D_{n+2}) \overset{d_{n+2}^{*}} \longrightarrow Hom(A_{n+2},D_{n+1}) \overset{d_{n+1}^{*}}
\longrightarrow Hom(A_{n+2},Ker\phi_{n}) $$ is
exact.
\noindent Since ${s}{a_{n+2}}+{\nu_{1}}{r_{1}} \in Hom(A_{n+2},D_{n+1})$ and \;$d_{n+1}^{*}({s}{a_{n+2}}+{\nu_{1}}{r_{1}})=d_{n+1}({s}{a_{n+2}}+{d_{n+1}}{\nu_{1}}{r_{1}})=({\theta_{n}}{a_{n+1}}-{\lambda_{n+1}}(0,r))a_{n+2}+{d_{n+1}}{\nu_{1}}{r_{1}}$\; by (**)
$=-\lambda_{n+1}(0,ra_{n+2})+{d_{n+1}}{\nu_{1}}{r_{1}}-{\nu_{0}}{r}{a_{n+2}}+{d_{n+1}}{\nu_{1}}{r_{1}}-{\nu_{0}}{t_{1}}{r_{1}}+{d_{n+1}}{\nu_{1}}{r_{1}}=0$.
\noindent So ${s}{a_{n+2}}+{\nu_{1}}{r_{1}} \in Im(d_{n+2})^{*}$. Thus there exists \;$s_{1}: A_{n+2} \longrightarrow D_{n+2}$ such that ${s}{a_{n+2}}+{\nu_{1}}{r_{1}}={d_{n+2}}{s_{1}}$.
\noindent Let $\theta_{n+2}: A_{n+2} \longrightarrow {D_{n+2} \oplus E_{1}}$;
$a \longrightarrow (s_{1}(a),-r_{1}(a))$.
\noindent Then \;${\lambda_{n+2}}{\theta_{n+2}}={\lambda_{n+2}}(s_{1},r_{1})=({d_{n+2}}{s_{1}}-{\nu_{1}}{r_{1}},{t_{1}}{r_{1}})=(sa_{n+2},ra_{n+2})={\theta_{n+1}}{a_{n+2}}$.\\

\noindent Let's find $\theta_{n+3}$. Since $Hom(A_{n+3},L^{n+1})$ is exact,
$Hom(A_{n+3},D_{n+3} \oplus E_{2}) \overset{\lambda_{n+3}^{*}} \rightarrow Hom(A_{n+3},D_{n+2} \oplus E_{1})
\overset{\lambda_{n+2}^{*}}
\rightarrow Hom(A_{n+3},D_{n+1} \oplus Kert_{0})$ is exact. Since
${\theta_{n+2}}{a_{n+3}} \in Hom(A_{n+3},D_{n+2} \oplus E_{1})$ and
$\lambda_{n+2}^{*}({\theta_{n+2}}{a_{n+3}})=\lambda_{n+2}({\theta_{n+2}}{a_{n+3}})=\theta_{n+1}({\theta_{n+2}}{a_{n+3}})=0$, ${\theta_{n+2}}{a_{n+3}} \in
Ker\lambda_{n+2}^{*}=Im\lambda_{n+3}^{*}$, therefore there exists
$\theta_{n+3} \in Hom(A_{n+3},D_{n+3} \oplus E_{2})$.\\

\noindent Let's find $\theta_{n+k}$ for $k > 3$. Since $Hom(A_{n+k},L^{n+1})$
is exact,\\
$ Hom(A_{n+k},D_{n+k} \oplus E_{k-1}) \overset{\lambda_{n+k}^{*}}
\longrightarrow Hom(A_{n+k},D_{n+k-1} \oplus E_{k-2})
\overset{\lambda_{n+k-1}^{*}} \longrightarrow
Hom(A_{n+k},$\\$D_{n+k-2} \oplus E_{k-3})$ is exact.\\
\noindent Since ${\theta_{n+k-1}}{a_{n+k}} \in Hom(A_{n+k},D_{n+k-1} \oplus
E_{k-2})$ and
$\lambda_{n+k-1}^{*}({\theta_{n+k-1}}{a_{n+k}})$\\$=\lambda_{n+k-1}{\theta_{n+k-1}}{a_{n+k}}
=\theta_{n+k-2}{a_{n+k-1}}{a_{n+k}}=0$,
${\theta_{n+k-1}}{a_{n+k}} \in
Ker\lambda_{n+k-1}^{*}$\\
$=Im\lambda_{n+k}^{*}$, therefore there exists $\theta_{n+k} \in
Hom(A_{n+k},D_{n+k} \oplus E_{k-1})$ such that
${\lambda_{n+k}}{\theta_{n+k}}={\theta_{n+k-1}}{a_{n+k}}$.
\end{proof}
\begin{lem} \label{1.29} Let $\mathcal{X}$ be closed under finite sum. If every module has an $\mathcal{X}$-preenvelope, then
every bounded complex has an $C(\mathcal{X^{*}})$-preenvelope.
\end{lem}
\begin{proof} Let $0 \longrightarrow X^{0} \overset{\phi^{0}} \longrightarrow E^{1}_{1} \overset{f^{1}} \longrightarrow E_{1}^{2}
 \longrightarrow ...$ be a minimal $\mathcal{X}$-preenvelope
resolvent of $X^{0}$. Let A be an $\mathcal{X^{*}}$-complex and $\beta\in Hom(X(0),A)$ where \;$X(0):... \longrightarrow 0 \longrightarrow
X^{0} \longrightarrow 0 \longrightarrow 0 \longrightarrow ...$. \\\\
Then we have the diagram:
\[\begin{diagram}
\node{...} \arrow{e} \node{0} \arrow{e} \arrow{s} \node{X^{0}}
\arrow{e} \arrow{s,r} {\phi^{0}} \node{0} \arrow{e} \arrow{s} \node{0} \arrow{e} \arrow{s} \node{...}\\
\node{...} \arrow{e} \node{0} \arrow{e}  \arrow{s} \node{E_{1}^{1}}
\arrow{e,t} {f^{1}} \arrow{s,r} {\lambda^{0}} \node{E_{1}^{2}}
\arrow{e,t} {f^{2}} \arrow{s,r} {\lambda^{1}} \node{E_{1}^{3}} \arrow{e,t} {f^{3}} \arrow{s,r} {\lambda^{2}} \node{...}\\
\node{...} \arrow{e} \node{A^{-1}} \arrow{e,t} {a^{-1}} \node{A_{0}}
\arrow{e,t} {a^{0}} \node{A^{1}} \arrow{e,t} {a^{1}}
\node{A^{2}} \arrow{e} \node{...}\\
\end{diagram}\]
\noindent where ${\lambda^{0}}{\phi^{0}}={\beta^{0}}$ and since
$\frac{E_{1}^{1}}{Im\phi^{0}} \longrightarrow A^{1}$ with
$x+Imf^{0}\mapsto {a^{0}}{\lambda^{0}(x)}$, there exists
$\lambda^{1}: E_{1}^{2} \longrightarrow A^{1}$ such that
${\lambda^{1}}{f^{1}}={a^{0}}{\lambda^{0}}$. Similarly we can define
$\lambda^{n} \in Hom(E_{1}^{n+1},A^{n})$ such that
${\lambda^{n}}{f^{n}}={a^{n-1}}{\lambda^{n-1}}$.

\noindent Let $X(n):... 0 \longrightarrow X^{0} \overset{\ell^{0}}
\longrightarrow X^{1} \overset{\ell^{1}} \longrightarrow X^{2}
\overset{\ell^{2}} \longrightarrow ... \longrightarrow X^{n}
\longrightarrow 0 \longrightarrow ...$. We will use an induction on
n to prove that $X(n)$ has an $C(\mathcal{X^{*}})$-preenvelope. At
the beginning we proved that $X(0)$ has an
$C(\mathcal{X^{*}})$-preenvelope. Suppose that
$D(n):...\longrightarrow 0 \longrightarrow D^{0} \overset{d^{0}}
\longrightarrow D^{1} \overset{d^{1}} \longrightarrow ....
\longrightarrow D^{n} \overset{d^{n}} \longrightarrow D^{n+1}
\longrightarrow ...$ be an $C(\mathcal{X^{*}})$-preenvelope of
$X(n)$ such that
\[\begin{diagram}
\node{X(n):...0} \arrow{e} \node{X^{0}} \arrow{e,t} {\ell^{0}}
\arrow{s,r} {\phi^{0}} \node{X^{1}...} \arrow{s,r} {\phi^{1}}
 \arrow{e} \node{X^{n}} \arrow{e,t} {\ell^{n}} \arrow{s,r}
{\phi^{n}} \node{0...} \\
\node{D(n):...0} \arrow{e} \node{D^{0}} \arrow{e,t} {d^{0}}
\node{D^{1}...} \arrow{e} \node{D^{n}} \arrow{e,t} {d^{n}} \node{D^{n+1}...}\\
\end{diagram}\]
Let $E^{\bullet}=... \longrightarrow 0 \longrightarrow E^{1}
\overset{t^{1}} \longrightarrow E^{2} \overset{t^{2}}
\longrightarrow E^{3} \overset{t^{3}} \longrightarrow ...$ be an
$C(\mathcal{X^{*}})$-preenvelope of $\underline{X^{n+1}}$(where
$X^{n+1}$ and $E^{1}$ is in nth place) where $t^{0}:X^{n+1}
\longrightarrow E^{1}$ is an $\mathcal{X}$-preenvelope of $X^{n+1}$.
Thus we have the morphism $\nu$ making the following commutative
diagram:
\[\begin{diagram}
\node{X(n)} \arrow{e} \arrow{s} \node{\underline{X^{n+1}}} \arrow{s}\\
\node{D(n)} \arrow{e,t,..} {\nu} \node{E^{\bullet}} \\
\end{diagram}\]
\noindent where ${t^{0}}{\ell^{n}}={\nu^{0}}{\phi^{n}}$...(*). Since $X(n)
\longrightarrow D(n)$ is an $C(\mathcal{X^{*}})$-preenvelope, we
have the diagram
\[\begin{diagram}
\node{D(n)=...} \arrow{e} \node{D^{n-1}} \arrow{e,t} {d^{n-1}}
\node{D^{n}} \arrow{e,t} {d^{n}} \arrow{s,r} {\nu^{0}}
\node{D^{n+1}} \arrow{e,t} {d^{n+1}} \arrow{s,r}
{\nu^{1}} \node{...}  \\
\node{E^{\bullet}=...} \arrow{e} \node{0} \arrow{e} \node{E^{1}}
\arrow{e,t} {t^{1}}
\node{E^{2}} \arrow{e,t} {t^{2}}  \node{...} \\
\end{diagram}\]...(1)
\noindent Let C be the mapping cone of $\nu:D(n) \longrightarrow E^{\bullet}$.
So $C=... \longrightarrow 0 \longrightarrow D^{0} \longrightarrow
D^{1} \longrightarrow ... \longrightarrow D^{n}
\overset{\lambda^{n}} \longrightarrow {D^{n+1} \oplus E^{1}}
\overset{\lambda^{n+1}} \longrightarrow {D^{n+2} \oplus E^{2}}
\overset{\lambda_{n+2}} \longrightarrow ...$ where
$$\lambda^{n}(x)=(d^{n}(x),\nu^{0}(x))$$ and for $k \geq 1$
$$\lambda^{n+k}(x,y)=(-d^{n+k}(x),\nu^{k}(x)+t^{k}(y)).$$

\noindent We show that $X(n+1) \rightarrow C$ is an
$C(\mathcal{X^{*}})$-preenvelope of $X(n+1)$.

\noindent Claim: Let \;$L^{k}=coker(X(k)\rightarrow D(k))$\; where $k \in
\{0,1,...,n\}$ and then
$$L^{0}:... \rightarrow coker(\phi^{0}) \rightarrow E_{1}^{1}
\overset{f^{1}} \rightarrow E_{1}^{2} \rightarrow ...,$$
$$L^{n-1}:... \longrightarrow{coker\phi^{n-1}} \overset {d_{1}^{n-1}}
\longrightarrow {D^{n}} \overset {d^{n}} \longrightarrow
{D^{n+1}}\longrightarrow...$$
$$L^{n}:...\longrightarrow{coker\phi^{n}} \overset {d_{1}^{n}} \longrightarrow
{D^{n+1}} \overset {d^{n+1}} \longrightarrow {D^{n+2}}
\longrightarrow...$$
$${L^{n+1}:...\rightarrow {coker\phi^{n}}
\rightarrow {D^{n+1}\oplus cokert^{0}} \overset {\lambda_{1}^{n+1}}
\rightarrow {D^{n+2} \oplus E^{2}} \overset{\lambda^{n+2}}
\rightarrow {D^{n+3} \oplus E^{3}}...}$$
\noindent where $\lambda_{1}^{n+1}(x,y)=\lambda^{n+1}(x,a)$ and $y=a+Imt^{0}$ for
$a\in E^{1}$. We claim that for every $S \in C(\mathcal{X^{*}})$,
$Hom(L^{k},S)$ is exact where $k \in \{0,1,...,n,n+1\}$. It is
understood easily that $Hom(L^{0},S)$ is exact. It is enough to
prove that if $Hom(L^{n},S)$ is exact, then \;$Hom(L^{n+1},S)$ is
exact. We have the following diagram between $L^{n}$ and $L^{n+1}$ :
\[\begin{diagram}
\node{...coker\phi^{n}} \arrow{e,t} {d_{1}^{n}} \arrow{s} \node{D^{n+1}} \arrow{e,t} {d^{n+1}} \arrow{s} \node{D^{n+2}} \arrow{e} \arrow{s}
\node{D^{n+3}...} \arrow{s}\\
\node{...coker\phi^{n}} \arrow{e}
\node{D^{n+1}\oplus cokert^{0}} \arrow{e,t}{\lambda_{1}^{n+1}}
\node{D^{n+2} \oplus E^{2}} \arrow{e,t}{\lambda^{n+2}} \node{D^{n+3}
\oplus E^{3}...}\\
\end{diagram}\]
\noindent And so, $\frac{L^{n+1}}{L^{n}}:... \rightarrow 0 \rightarrow 0 \rightarrow
... \rightarrow 0 \rightarrow cokert^{0} \rightarrow E^{2}
\rightarrow ...$ is a resolution of $cokert^{0}$, and hence
$Hom(\frac{L^{n+1}}{L^{n}},S)$ is exact where S is an
$\mathcal{X^{*}}$-complex. Since $0 \rightarrow L^{n} \rightarrow
L^{n+1} \rightarrow \frac{L^{n+1}}{L^{n}} \rightarrow 0$ is split
exact, we have an exact complex \;$0 \rightarrow Hom(\frac{L^{n+1}}{L^{n}},S) \rightarrow
Hom(L^{n+1},S) \rightarrow Hom(L^{n},S)\rightarrow 0$. Therefore
$Hom(\frac{L^{n+1}}{L^{n}},S)$ and $Hom(L^{n},S)$ are exact , so
$Hom(L^{n+1},S)$ is exact.\\
\noindent Let $A \in C(\mathcal{X^{*}})$. Then we have the diagram
as follow:
\[\begin{diagram}
\node{...X^{0}} \arrow{e,t} {\ell^{0}}
\arrow{s,r} {\phi^{0}} \node{...X^{n}} \arrow{e,t} {\ell^{n}}
\arrow{s,r} {\phi^{n}} \node{X^{n+1}} \arrow{e,t} {\ell^{n+1}}
\arrow{s,r} {(0,t^{0})} \node{0...}
 \arrow{s}\\
\node{...D^{0}} \arrow{e,t} {d^{0}} \arrow{s,r}
{\theta^{0}} \node{...D^{n}} \arrow{e,t} {d^{n}} \arrow{s,r}
{\phi^{n}} \arrow{s,r} {\theta^{n}} \node{D^{n+1} \oplus E^{1}}
\arrow{e,t} {\lambda^{n+1}} \arrow{s,r} {\theta^{n+1}} \node{D^{n+2}
\oplus E^{2}...}
\arrow{s,r} {\theta^{n+2}}\\
\node{...A^{-1}} \arrow{e,t} {a^{-1}} \node{A^{0}} \arrow{e,t}
{a^{0}} \node{...A^{n}} \arrow{e,t} {a^{n}} \node{A^{n+1}...}  \\
\end{diagram}\]
\noindent where the $\beta^{i}:X^{i} \rightarrow A^{i}$ are maps such that
${\theta^{0}}{\phi^{0}}={\beta^{0}}$, ...,
${\theta^{n}}{\phi^{n}}={\beta^{n}}$.

\noindent First, we will find $\theta^{n+1}$ such that
${\theta^{n+1}}{\lambda^{n}}={a^{n}}{\theta^{n}}$ and
${\theta^{n+1}}{(0,t^{0})}={\beta^{n+1}}$. Moreover
${\lambda^{n}}{\phi^{n}}={(0,t^{0})}{\ell^{n}}$ by (*). Since
$\beta^{n+1}:X^{n+1} \longrightarrow A^{n+1}$ and $A^{n+1} \in
\mathcal{X}$ and $t^{0}:X^{n+1} \longrightarrow E^{1}$ is an
$\mathcal{X}$-preenvelope, there exists $r: E^{1} \longrightarrow
A^{n+1}$ such that $\beta^{n+1}={r}{t^{0}}$.

\noindent Since $Hom(L^{n-1},A^{n+1})$ is exact where $A^{n+1} \in
\mathcal{X}$, $... \rightarrow Hom(D^{n+1},A^{n+1})$\\
$ \overset{(d^{n+1})^{*}}
\rightarrow Hom(D^{n},A^{n+1}) \overset{({d_{1}}^{n-1})^{*}}
\rightarrow Hom(coker\phi^{n-1},A^{n+1}) \rightarrow ...$\; is exact,
that is, $Ker({d_{1}}^{n-1})^{*}=Im(d^{n})^{*}$.

\noindent Claim: ${a^{n}}{\theta^{n}}-{r}{\nu^{0}}  \in
Ker({d_{1}}^{n-1})^{*}$.

\noindent Since $({d_{1}}^{n-1})^{*}({a^{n}}{\theta^{n}}-{r}{\nu_{0}})=({a^{n}}{\theta^{n}}-{r}{\nu^{0}}){d_{1}}^{n-1}={a^{n}}{\theta^{n}}{d_{1}}^{n-1}-{r}{\nu_{0}}{d_{1}}^{n-1}={a^{n}}{\theta^{n}}{d_{1}}^{n-1}$ by diagram (1). Since
\;${\theta^{n}}{d_{1}^{n-1}}(x+Im\phi^{n-1})={a^{n-1}}{\theta^{n-1}}(x)$\;
where $x \in D^{n-1}$, ${a^{n}}{\theta^{n}}{{d_{1}}^{n-1}}=0$. So ${a^{n}}{\theta^{n}}-{r}{\nu^{0}} \in
{Ker(d_{1}}^{n-1})^{*}=Im({d^{n}})^{*}$ and hence there exists $s
\in Hom(D^{n+1},A^{n+1})$ such that
${a^{n}}{\theta^{n}}-{r}{\nu^{0}}={s}{d^{n}}$ ...(**)

\noindent Let $\theta^{n+1}:D^{n+1} \oplus E^{1} \rightarrow A^{n+1}$ with
$(a,b) \mapsto s(a)+r(b)$. Then
$\theta^{n+1}(0,t^{0})(x)=\theta^{n+1}(0,t^{0}(x))=r{t^{0}(x)}=\beta^{n+1}(x)$
where $x \in X^{n+1}$. Therefore
$\theta^{n+1}(0,t^{0})=\beta^{n+1}$.

\noindent Moreover, \;${\theta^{n+1}}{\lambda^{n}}(x)={\theta^{n+1}}(\lambda^{n}(x))=\theta^{n+1}(d^{n}(x),\nu^{0}(x))={a^{n}}{\theta^{n}}(x)-r{\nu^{0}}(x)+r{\nu^{0}}(x) $\; by (**)
$$={a^{n}}{\theta^{n}}(x).$$
\noindent Therefore ${\theta^{n+1}}{\lambda^{n}}={a^{n}}{\theta^{n}}$.

\noindent Secondly, let's find $\theta^{n+2}$. We know that
$\lambda^{n+1}(0,t^{0})=0$. Since $\alpha: cokert^{0}
\longrightarrow A^{n+2}$ with $x+Imt^{0}\mapsto -{a^{n+1}}{r(x)}$ is
a map and $t_{1}^{1}:cokert^{0} \longrightarrow E^{2}$ is an
$\mathcal{X}$-preenvelope where ${t_{1}^{1}}{\pi}={t^{1}}$, there
exists $r^{1}:E^{2} \longrightarrow A^{n+2}$ such that
${r^{1}}{t^{1}}=\alpha$. Since $Hom(L^{n},A^{n+2})$ is exact,
$$ Hom(D^{n+2},A^{n+2}) \overset{{d^{n+1}}^{*}} \rightarrow Hom(D^{n+1},A^{n+2}) \overset{{d^{n}}^{*}}
\rightarrow Hom(D^{n},A^{n+2})$$ is exact, and so
\;$Im((d^{n+1})^{*})=Ker(d^{n})^{*}$.

\noindent Claim: $-{a^{n+1}}{s}+{r^{1}}{\nu^{1}} \in Ker(d^{n})^{*}$.

$(d^{n})^{*}(-{a^{n+1}}{s}+{r^{1}}{\nu^{1}})=(-{a^{n+1}}{s}+{r^{1}}{\nu^{1}})d^{n}=-{a^{n+1}}{s}d^{n}+{r^{1}}{\nu^{1}}d^{n}=-{a^{n+1}}({a^{n}}{\theta^{n}}-{r}{\nu^{0}})+{r^{1}}{\nu^{1}}d^{n}$ by(**)\\
$={a^{n+1}}{r}{\nu^{0}}+{r^{1}}{\nu^{1}}d^{n}={a^{n+1}}{r}{\nu^{0}}+{r^{1}}{t^{1}}\nu^{0}$by diagram (1)\\
$={a^{n+1}}{r}{\nu^{0}}+{r^{1}}{t^{1}_{1}}{\pi}\nu^{0}={a^{n+1}}{r}{\nu^{0}}+\alpha{\pi}\nu^{0}={a^{n+1}}{r}{\nu^{0}}-{a^{n+1}}{r}\nu^{0}=0$.
\noindent Therefore $-{a^{n+1}}{s}+{r^{1}}{\nu^{1}} \in {Im(d^{n+1})}^{*}$. So
there exists $s^{1} \in Hom(D^{n+2},A^{n+2})$ such that
$-{a^{n+1}}{s}+{r^{1}}{\nu^{1}}={s^{1}}{d^{n+1}}.$

\noindent Let $\theta^{n+2}: {D^{n+2} \oplus E^{2}} \longrightarrow A^{n+2}$
with $(x,y) \mapsto s^{1}(x)+r^{1}(y)$

\noindent Then \; ${\theta^{n+2}}{\lambda^{n+1}}(x,y)={\theta^{n+2}}(-d^{n+1}(x),\nu^{1}(x)+t^{1}(y))=s^{1}(-d^{n+1}(x))+r^{1}(\nu^{1}(x)+t^{1}(y))=-s^{1}d^{n+1}(x)+r^{1}\nu^{1}(x)+r^{1}t^{1}(y)=a^{n+1}s(x)-r^{1}\nu^{1}(x)+r^{1}\nu^{1}(x)+r^{1}t^{1}(y)=a^{n+1}s(x)+r^{1}t^{1}(y)={a^{n+1}}{\theta^{n+1}}(x,y).$

\noindent Let's find $\theta^{n+3}$. Since $Hom(L^{n+1},A^{n+3})$ is exact,
$Hom(D^{n+3} \oplus E^{3},A^{n+3})$\\$ \overset{{(\lambda^{n+2})}^{*}}
\rightarrow Hom(D^{n+2} \oplus E^{2},A^{n+3})
\overset{{(\lambda_{1}^{n+1})}^{*}} \rightarrow Hom(D^{n+1} \oplus
coker(t^{0}),A^{n+3})$ is exact. Since ${a^{n+2}}{\theta^{n+2}} \in
Hom(D^{n+2} \oplus E^{2},A^{n+3})$ and
$$({\lambda_{1}^{n+1}})^{*}({a^{n+2}}{\theta^{n+2}})(x,y)={a^{n+2}}{\theta^{n+2}}\lambda^{n+1}(x,a)
={a^{n+2}}{a^{n+1}}\theta^{n+1}(x,a)=0$$ where $y=a+Im(t^{0})$ and
$a\in E^{1}$, we have that $${a^{n+2}}{\theta^{n+2}} \in
Im({\lambda^{n+2}})^{*}$$, therefore there exists $\theta^{n+3} \in
Hom(D^{n+3} \oplus E^{3},A^{n+3})$ such that
$${a^{n+2}}{\theta^{n+2}}={\theta^{n+3}}{\lambda^{n+2}}.$$

\noindent Let's find $\theta^{n+k}$ for $k > 3$. Since
$Hom(L^{n+1},A^{n+k})$ is exact, $ Hom(D^{n+k} \oplus
E^{k},A^{n+k}) \overset{({\lambda^{n+k-1}})^{*}} \rightarrow
Hom(D^{n+k-1} \oplus E^{k-1},A^{n+k})
\overset{({\lambda^{n+k-2}})^{*}} \rightarrow Hom(D^{n+k-2} \oplus
E^{k-2},A^{n+k})$ is exact.

\noindent Since ${a^{n+k-1}}{\theta^{n+k-1}} \in Hom(D^{n+k-1}
\oplus E^{k-1},A^{n+k})$ and
$$({\lambda^{n+k-2}})^{*}({a^{n+k-1}}{\theta^{n+k-1}})={a^{n+k-1}}{\theta^{n+k-1}}\lambda^{n+k-2}
={a^{n+k-1}}{a^{n+k-2}}\theta^{n+k-2}$$$=0$,
\;${a^{n+k-1}}{\theta^{n+k-1}} \in Im({\lambda^{n+k-1}})^{*}$,and so
there exists $\theta^{n+k} \in Hom(D^{n+k} \oplus E^{k},A^{n+k})$
such that
${a^{n+k-1}}{\theta^{n+k-1}}={\theta^{n+k}}{\lambda^{n+k-1}}$.
\end{proof}

\noindent Again using the tecnique of Alina Iacob we can give the following result.
\begin{prop} Let $\mathcal{X}$ be a class of modules with closed under direct sum. If every
module has an $\mathcal{X}$-precover, then every bounded above
complex has a $C(\mathcal{X^{*}})$-precover.
\end{prop}

\begin{proof} Let $X(n): ... \rightarrow 0 \rightarrow X_{n} \rightarrow X_{n-1} \rightarrow ... \rightarrow
X_{1} \rightarrow X_{0} \rightarrow 0$ and
$X=\underrightarrow{lim}X(n)$. (Then $X=... \rightarrow X_{2}
\rightarrow X_{1} \rightarrow X_{0} \rightarrow 0$).

\noindent By Lemma 2.3, we know that $X(n)$ has a
$C(\mathcal{X^{*}})$-precover such that $0 \rightarrow L^{n}
\rightarrow D(n) \rightarrow X(n)$ where $L^{n}$ is the kernel of
($D(n) \rightarrow X(n)$) and $Hom(S,L^{n})$ is exact for any
$\mathcal{X^{*}}$-complex S.

\noindent Let $L=\underrightarrow{lim} L^{n}$ and so
$L_{k}=L_{k}^{k}$ for any $k \geq 0$.

\noindent Using the proof of Lemma 2.3 we can say that $Hom(S,L_{k+1})
\rightarrow Hom(S,L_{k}) \rightarrow Hom(S,L_{k-1})$ is exact, that
is, $Hom(S,L)$ is exact. We show that $\underrightarrow{lim} D(n)
\rightarrow X$ is a $C(\mathcal{X^{*}})$-precover.

\noindent Let $D=\underrightarrow{lim} D(n)$ and $\beta \in
Hom(A,X)$ where A is an $\mathcal{X^{*}}$-complex. Again using the
proof of Lemma 2.3 we can find $\phi \in Hom(D,X)$. Then
$L^{0}=Ker\phi_{0}$, $L^{1}=Ker\phi_{1}$, $L^{2}=Ker\phi_{2}$, ...

\noindent We have the diagram;

\[\begin{diagram}
\node{...} \arrow{e} \node{A_{2}} \arrow{e,t} {a_{2}} \arrow{s}
\node{A_{1}} \arrow{e,t} {a_{1}} \arrow{s} \node{A_{0}} \arrow{e,t}
{a_{0}} \arrow{s} \node{A_{-1}} \arrow{e} \node{...} \\
\node{...} \arrow{e} \node{D_{2}} \arrow{e,t} {d_{2}} \arrow{s,r}
{\phi_{2}} \node{D_{1}}
\arrow{e,t} {d_{1}} \arrow{s,r} {\phi_{1}} \node{D_{0}} \arrow{e} \arrow{s,r} {\phi_{0}} \node{0} \\
\node{...} \arrow{e} \node{X_{2}}
\arrow{e,t} {\ell_{2}} \node{X_{1}} \arrow{e,t} {\ell_{1}} \node{X_{0}} \arrow{e} \node{0} \\
\end{diagram}\]

\noindent where ${\phi_{0}}{\alpha_{0}}={\beta_{0}}$,
${\phi_{1}}{r_{1}}={\beta_{1}}$, ...,
${\phi_{i}}{r_{i}}={\beta_{i}}$, ....

\noindent Then
$\phi_{0}({\alpha_{0}}{a_{1}}-{d_{1}}{r_{1}})={\phi_{0}}{\alpha_{0}}{a_{1}}-{\phi_{0}}{d_{1}}{r_{1}})$
$={\beta_{0}}{a_{1}}-{\ell_{1}}{\phi_{1}}{r_{1}}$
$={\beta_{0}}{a_{1}}-{\ell_{1}}{\beta_{1}}=0.$

\noindent Thus ${\alpha_{0}}{a_{1}}-{d_{1}}{r_{1}} \in Hom(A_{1},
Ker\phi_{0}=L_{0})$. Since $Hom(A_{1},L)$ is exact, there is $s_{1}
\in Hom(A_{1},L_{1})$ such that
${\alpha_{0}}{a_{1}}-{d_{1}}{r_{1}}={d_{1}}{s_{1}}$. Let
$\alpha_{1}=s_{1}+r_{1}$. Then
${\phi_{1}}{\alpha_{1}}=\phi_{1}(s_{1}+r_{1})={\phi_{1}}{s_{1}}+{\phi_{1}}{r_{1}}=0$

\noindent If we continue in this way, then we can find $\alpha_{i}
\in Hom(A_{i},D_{i})$ such that
${d_{i}}{\alpha_{i}}={\alpha_{i-1}}{a_{i}}$. Thus $D \rightarrow X$
is a $C(\mathcal{X^{*}})$-precover.
\end{proof}
\noindent Using the modification to complexes of Corollary 5.2.7 in \cite{EJ} we can give the following result.

\begin{cor} Let $\mathcal{X}$ be a class with closed under  direct sum and $C(\mathcal{X^{*}})$ be closed under direct limit. If
every module has an $\mathcal{X}$-precover, then every bounded complex
has a $C(\mathcal{X^{*}})$-cover.
\end{cor}

\begin{thm} Let $\mathcal{X}$ be closed under direct sum and extension
closed . Let $C(\mathcal{X^{*}})$
be closed under inverse and direct limit. If every module has an
$\mathcal{X}$-cover, then every complex has a
$C(\mathcal{X^{*}})$-cover.
\end{thm}

\begin{proof} First we will show that every complex has
$C(\mathcal{X^{*}})$-precover if $\mathcal{X}$ is a class of
modules with extension closed and closed under direct sum and
$C(\mathcal{X^{*}})$ is closed under inverse limit. And then modifying
Corollary 5.2.7 in \cite{EJ} we can find a
$C(\mathcal{X^{*}})$-cover if $C(\mathcal{X^{*}})$ is closed under
direct limit.

\noindent Let $X: ... \rightarrow X_{2} \rightarrow X_{1} \rightarrow X_{0}
\rightarrow X_{-1} \rightarrow...$ and let $Y^{n}: ... \rightarrow
X_{1} \rightarrow X_{0} \rightarrow X_{-1} \rightarrow
...\rightarrow X_{-n} \rightarrow 0 \rightarrow ...$. So
$\underleftarrow{\lim} Y^{n}=X$.

\noindent We see that $Y^{n}$ has a
$C(\mathcal{X^{*}})$-precover $D^{n}:... \rightarrow D_{1} \rightarrow
D_{0} \rightarrow D_{-1} \rightarrow ... \rightarrow D_{-n}
\rightarrow 0 \rightarrow...$ such that $$D_{-n}=E_{-n}^{0},$$
$$D_{-n+1}=E_{-n}^{1} \oplus E_{-n+1}^{0},$$ $$D_{-n+2}=E_{-n}^{2} \oplus E_{-n+1}^{1} \oplus
E_{-n+2}^{0},$$ ...

\noindent where $... \rightarrow E_{k}^{2} \rightarrow E_{k}^{1}
\rightarrow E_{k}^{0} \overset{\phi_{k}} \rightarrow X_{k}
\rightarrow 0$ is a minimal $\mathcal{X}$-precover resolvent of
$X_{k}$ for $k \geq -n$.

\noindent Then $Y^{n+1}=... \rightarrow X_{0} \rightarrow X_{1}
\rightarrow ... \rightarrow X_{-n} \rightarrow X_{-n-1} \rightarrow
0 \rightarrow ...$ and by the proof of Proposition 2.5  the $C(\mathcal{X^{*}})$-precover of $Y^{n+1}$,
$D^{n+1}=...\rightarrow E_{-n-1}^{1} \oplus D_{-n} \rightarrow
E_{-n-1}^{0} \rightarrow 0 \rightarrow...$.

\noindent Let $T^{n}=Ker(\phi: D^{n} \rightarrow Y^{n})$ and then again by the proof of Proposition 2.5, 
$T^{n+1}=... E_{-n-1}^{2} \oplus E_{-n}^{1} \oplus Ker\phi_{-n+1}
\rightarrow E_{-n-1}^{1} \oplus Ker\phi_{-n} \rightarrow
Ker\phi_{-n-1} \rightarrow 0$,

\noindent $T^{n}=... \rightarrow E_{-n}^{1} \oplus Ker\phi_{-n+1}
\rightarrow Ker\phi_{-n} \rightarrow 0$.

\noindent Then $Ker(T^{n+1} \rightarrow T^{n})=... \rightarrow
E_{-n-1}^{2} \rightarrow E_{-n-1}^{1} \rightarrow Ker\phi_{-n-1}
\rightarrow 0$ is the kernel of a $C(\mathcal{X^{*}})$-cover of
$\underline{Y_{-n-1}}$. By Wakamatsu's Lemma $Hom(S,T^{n+1}) \rightarrow
Hom(S,T^{n}) \rightarrow 0$ is epic for any
$\mathcal{X^{*}}$-complex S. Since we have the exact sequence $0
\rightarrow T^{n} \rightarrow D^{n} \rightarrow Y^{n}$ and $D^{n}$
is a $C(\mathcal{X^{*}})$-cover of $Y^{n}$, we have the exact sequence $0 \rightarrow
Hom(S,T^{n}) \rightarrow Hom(S,D^{n}) \rightarrow Hom(S,Y^{n})
\rightarrow 0$ for any
$\mathcal{X^{*}}$-complex S.

\noindent By the modification of Theorem 1.5.13 and 1.5.14 in \cite{EJ} we have the following exact sequence; $0 \rightarrow
Hom(S,\underleftarrow{\lim} T^{n}) \rightarrow
Hom(S,\underleftarrow{\lim} D^{n}) \rightarrow
Hom(S,\underleftarrow{\lim} Y^{n}) \rightarrow 0$. So,
$\underleftarrow{\lim} D^{n}$ is a $C(\mathcal{X^{*}})$-precover of X.
\end{proof}
\noindent Dually, we can say the following theorems without proof;

\begin{prop} Let $\mathcal{X}$ be closed under direct sum. If every module
has an $\mathcal{X}$-preenvelope, then every right bounded complex
has a $C(\mathcal{X^{*}})$-preenvelope.
\end{prop}

\noindent Using the modification to complexes of Corollary 6.3.5 in \cite{EJ} we can give the following result.

\begin{cor} Let $\mathcal{X}$ be a class with closed under direct sum and summands and $C(\mathcal{X^{*}})$ be closed under inverse limit. If
every module has an $\mathcal{X}$-preenvelope, then every bounded complex
has a $C(\mathcal{X^{*}})$-envelope.
\end{cor}

\begin{thm} Let $\mathcal{X}$ be closed under direct sum and summands and extension
closed and $C(\mathcal{X^{*}})$ be closed under inverse and direct limit. If
every module has an $\mathcal{X}$-envelope, then every complex has a
$C(\mathcal{X^{*}})$-envelope.

\end{thm}

\end{document}